\documentclass[]{article}

\usepackage[utf8]{inputenc}
\usepackage[margin=1in]{geometry}

\usepackage{amsthm, amsfonts, amsmath, makecell, tikz, 
 graphicx, colortbl}
\usepackage{multirow, booktabs, bigstrut}
\usepackage{adjustbox}

\usepackage{graphicx}

\usepackage{float}
\usepackage{color}
\usepackage{tikz,pgfplots}
\usepackage{tikz-cd}
\usepackage{verbatim}
\usepackage{enumerate}
\usepackage{physics}
\usepackage{braket}

\usepackage{bm}

\usepackage{enumerate, amssymb}

\usetikzlibrary{graphs,graphs.standard,arrows.meta,calc,intersections}

\usepackage[]{algorithm, algorithmic}

\usepackage{tabularx}

\theoremstyle{plain}
\newtheorem{theorem}{Theorem}[section]
\newtheorem{lemma}[theorem]{Lemma}

\newtheorem{proposition}[theorem]{Proposition}

\newtheorem{remark}[theorem]{Remark}

\theoremstyle{definition}
\newtheorem{definition}[theorem]{Definition}
\newtheorem{example}[theorem]{Example}

\DeclareMathOperator{\Betti}{Betti}

\renewcommand{\epsilon}{\varepsilon}

\newcommand{\N}{{\mathbb N}}

\usepackage{stmaryrd }

\title{Longest and Shortest Factorizations in Embedding Dimension Three}

\author{Baian Liu, JiaYan Yap}

\begin{document}
	\maketitle
	
	\begin{abstract}
		For a numerical monoid $\langle n_1, \dots, n_k \rangle$ minimally generated by $n_1, \dots, n_k \in \N$ with $n_1 < \cdots < n_k$, the longest and shortest factorization lengths of an element $x$, denoted as $L(x)$ and $\ell(x)$, respectively, follow the identities $L(x+n_1) = L(x) + 1$ and $\ell(x+n_k) = \ell(x) + 1$ for sufficiently large elements $x$. We characterize when these identities hold for all elements of numerical monoids of embedding dimension three.
	\end{abstract}

\section{Introduction}

\indent\indent Let $\N$ be the set of natural numbers including 0. We will study factorizations on \textbf{numerical monoids}, which are cofinite submonoids of $\N$. Numerical monoids can be presented using $n_1, \dots, n_k \in \N$ such that $\gcd(n_1, \dots, n_k) = 1$ as generators:
\[
	\langle n_1, \dots, n_k \rangle = \{a_1n_1+ \cdots + a_kn_k \mid a_1, \dots, a_k \in \N \}.
\]
The \textbf{embedding dimension} of a numerical monoid $M$ is the minimum $k \in \N$ such that there exists $n_1, \dots, n_k \in \N$ so that $M = \langle n_1, \dots, n_k \rangle$. 

Each element of $\langle n_1, \dots, n_k \rangle$ can be obtained using different $\N$-linear combinations of the generators $n_1, \dots, n_k$. We collect this information in a set and study certain invariants on it. 

\begin{definition}
	Let $\langle n_1, \dots, n_k \rangle$ be a numerical monoid minimally generated by $n_1, \dots, n_k \in \N$ with $n_1 < \cdots < n_k$. Take an element $x \in \langle n_1, \dots, n_k \rangle$. We define the \textbf{set of factorizations of $x$} to be
	\[
		\mathsf{Z}(x) = \{(a_1, \dots, a_k) \in \N^k \mid a_1n_1 + \cdots + a_kn_k = x \}. 
	\]
	For each factorization $(a_1, \dots, a_k) \in \mathsf{Z}(x)$, we define its \textbf{length} to be
	\[
		\abs{(a_1, \dots, a_k)} = a_1 + \cdots + a_k. 
	\]
	The \textbf{set of factorization lengths} for $x$ are then defined to be
	\[
		\mathcal{L}(x) = \{\abs{(a_1, \dots, a_k)} \mid (a_1, \dots, a_k) \in \mathsf{Z}(x) \}.
	\]
	Then the \textbf{maximum factorization length of $x$ }and the \textbf{minimum factorization length of $x$ }is defined as
	\[
		L(x) = \max \mathcal{L}(x) \text{\quad and \quad} \ell(x) = \min \mathcal{L}(x),
	\]
	respectively.
\end{definition}

The invariants $L(x)$ and $\ell(x)$ satisfy a linear-like formula given that $x$ is a sufficiently large element of the numerical monoid, as provided by the following result. 

\begin{theorem}\cite[Theorems 4.2 and 4.3]{barron2017set}
	Let $\langle n_1, \dots, n_k \rangle$ be a numerical monoid with $n_1 < \cdots < n_k$. The following statements hold.
	\begin{enumerate}
		\item 	 If $x > (n_1-1)n_k-n_1$, then $L(x+n_1) = L(x) + 1$.
		\item 	If $x > (n_{k}-1)n_{k-1} - n_k$, then $\ell(x+n_k) = \ell(x) + 1$.
	\end{enumerate}
\end{theorem}

These formulas hold for sufficiently large elements, but for small elements, the formulas can still hold in certain numerical monoids. We want to find some conditions for a numerical monoid $M$ under which the equality $L(x+n_1) = L(x) + 1$ holds for all $x \in M$. Similarly, we want to see when $\ell(x+n_k) = \ell(x) + 1$ holds for all $x \in M$. This question was posed in \cite[p. 334]{chapman2018factoring}.

If we first restrict ourselves to a numerical monoid $M$ of embedding dimension two, we find that the equalities $L(x+n_1) = L(x) + 1$ and $\ell(x+n_k) = \ell(x) + 1$ always hold for all $x \in M$ \cite[p. 334]{chapman2018factoring}. 

\begin{proposition}
	Let $M = \langle n_1, n_2 \rangle$ be a numerical monoid of embedding dimension two with $n_1 < n_2$. Then $L(x+n_1) = L(x) + 1$ and $\ell(x+n_2) = \ell(x) + 1$ for all $x \in M$. 
\end{proposition}

For this reason, we restrict ourselves to numerical monoids of embedding dimension three.

\section{Betti Elements}\label{Sec:Betti}
\indent\indent An important invariant of numerical monoids is the set of Betti elements. We use this set to better describe numerical monoids of embedding dimension three. 

\begin{definition}
	Let $M$ be a numerical monoid. Let $x \in M$ be an element and we define $\nabla_x$ as a graph with vertices $\mathsf{Z}(x)$ and two distinct vertices $v, v' \in \mathsf{Z}(x)$ are connected by an edge if and only if $v \cdot v' \neq 0$. We call $x$ a \textbf{Betti element} if $\nabla_x$ is not connected, and we will use $\textup{Betti}(M)$ to denote the set of Betti elements of $M$. 
\end{definition}

For a numerical monoid $\langle n_1, n_2, n_3 \rangle$ of embedding dimension three with $n_1 < n_2 < n_3$, we can define for each $i = 1,2,3$ the numbers
\[
	c_i = \min \{k \in \N \setminus \{0\} \mid \text{there exists $(a_1, a_2, a_3) \in \mathsf{Z}(kn_i)$ such that $a_i = 0$} \}.
\]

Using these numbers $c_1, c_2, c_3$, we can describe the set of Betti elements of a numerical monoid of embedding dimension three. 

\begin{proposition}\cite[p. 582]{garcia2018delta}\label{Prop:FormOfBettiElts}
	For a numerical monoid $M = \langle n_1, n_2, n_3 \rangle$ of embedding dimension three, we have
	\[
		\textup{Betti}(M) = \{c_1n_1, c_2n_2, c_3n_3\}. 
	\]
\end{proposition}

The elements in the set above are not necessarily distinct. The number of Betti elements of a numerical monoid of embedding dimension three can be determined by the form of the generators. 

\begin{theorem}\cite[p. 583]{garcia2018delta}\label{Thm:2BettiForm}
	Let $M = \langle n_1, n_2, n_3 \rangle$ be a numerical monoid of embedding dimension three. Then $\abs{\textup{Betti}(M)} \leq 2$ if and only if 
	\[
		\{n_1, n_2, n_3\} = \{am_1, am_2, bm_1 + cm_2 \}
	\]
	for some $a,b,c, m_1, m_2 \in \N$ such that $ 1 < m_1 < m_2$, $\gcd(m_1, m_2) = 1$, $a\geq 2$, $ b+c \geq 2$, and $\gcd(a, bm_1+cm_2) = 1$. 
\end{theorem}

\begin{remark}
	In \cite[p. 583]{garcia2018delta}, the case of one Betti element and two Betti elements are separated. The case of two Betti elements is listed as above. In the case when $\abs{\Betti(M)} = 1$, we have $\{n_1, n_2, n_3\} = \{s_2s_3, s_1s_3, s_1s_2\}$ for some positive pariwise coprime integers $s_1, s_2, s_3$. This has the same form as above by setting $a = s_3, m_1 = s_2, m_2 = s_1, b=s_1, c = 0$. 
\end{remark}

In particular, if $M$ is a numerical monoid of embedding dimension three and the three minimal generators are pairwise coprime, then $\abs{\textup{Betti}(M)} = 3$. However, the converse is not true in general. 

Due to how $c_i$ is defined, there exist $r_{12}, r_{13}, r_{21}, r_{23}, r_{31}, r_{32} \in \N$ such that
\begin{align*}
	c_1n_1 = r_{12}n_2 + r_{13}n_3,\\
	c_2n_2 = r_{21}n_1 + r_{23}n_3,\\
	c_3n_3 = r_{31}n_1 + r_{32}n_2.
\end{align*}

In \cite[Lemma 1]{rosales2004numerical}, it was shown that each $r_{ij} > 0$ if the generators are pairwise coprime. Moreover, the following lemma establishes a relationship between these coefficients. 

\begin{lemma}\cite[Lemma 3]{rosales2004numerical}\label{Lemma:ColumnAdd}
	For every $\{i,j,k\} = \{1,2,3\}$,
	\[
		c_i = r_{ji} + r_{ki}
	\]
\end{lemma}

If the numerical monoid $\langle n_1, n_2, n_3 \rangle$ has $\gcd(n_i,n_j) = 1$ for $i \neq j$, then the coefficients $r_{ij}$ are unique. The numerical monoid can also be recovered from these invariants.

\begin{theorem}\cite[Theorem 9]{rosales2004numerical}\label{Thm:CoprimeBijection}
	Let $X$ be the set consisting of all numerical monoids of embedding dimension three having the form $\langle n_1, n_2, n_3 \rangle$ such that $\gcd(n_i, n_j) = 1$ if $i \neq j$ and $n_1 < n_2 < n_3$. Let $Y$ be the set of $3 \times 3$ matrices of the form 
	\[
		\begin{pmatrix}
		0 & a_{12} & a_{13} \\
		a_{21} & 0 & a_{23}\\
		a_{31} & a_{32} & 0
		\end{pmatrix},
	\]
	where $a_{12}, a_{13}, a_{21}, a_{23}, a_{31}, a_{32} \in \N \setminus \{0\}$, and the numbers $A_1 = a_{12}a_{13} + a_{12}a_{23} + a_{13}a_{32}, A_2 = a_{13}a_{21} + a_{21}a_{23} + a_{23}a_{31}$, and $A_3 = a_{12}a_{31} + a_{21}a_{32} + a_{31}a_{32}$ have the properties that $A_1 < A_2 < A_3$ and $\gcd(A_i, A_j) = 1$ if $i \neq j$. There exists a bijection between $X$ and $Y$

given by the assignments 
\[
	\langle n_1, n_2, n_3 \rangle \mapsto \begin{pmatrix}
	0 & r_{12} & r_{13} \\
	r_{21} & 0 & r_{23}\\
	r_{31} & r_{32} & 0
	\end{pmatrix} \text{ and } \begin{pmatrix}
	0 & a_{12} & a_{13} \\
	a_{21} & 0 & a_{23}\\
	a_{31} & a_{32} & 0
	\end{pmatrix} \mapsto \langle A_1, A_2, A_3 \rangle . 
\]
\end{theorem}

This bijection can be extended to all numerical monoids of embedding dimension three with three Betti elements. The proofs of \cite[Lemmas 2, 3, 4, and 5]{rosales2004numerical} do not rely on the fact that the generators are pairwise coprime. However, they do rely on the result of \cite[Lemma 1]{rosales2004numerical}, which was proven using the fact that the generators are pairwise coprime. We provide here the result of \cite[Lemma 1]{rosales2004numerical} under the weaker assumption that there are three Betti elements. 

\begin{lemma}
	Let $M = \langle n_1, n_2, n_3 \rangle$ be a numerical monoid of embedding dimension three such that $\abs{\Betti(M)} = 3$. Then $r_{12}, r_{13}, r_{21}, r_{23}, r_{31}, r_{32}$ are positive integers.
\end{lemma}

\begin{proof}
	We suppose for a contradiction that $r_{ij} = 0$ for some $i, j$. Then for $\{i,j,k\} = \{1,2,3\}$, we have 
	\[
		c_i n_i = r_{ik} n_k. 
	\]
	By the minimality of $c_k$, we obtain $r_{ik} \geq c_k$. Since $\abs{\Betti(M)} = 3$, we cannot have $c_i n_i = c_k n_k$. Thus, $r_{ik} > c_k$. Subtracting $c_k n_k = r_{ki} n_i + r_{kj} n_j$ from $c_i n_i = r_{ik} n_k$ and rearranging, we get that
	\[
		(c_i - r_{ki}) n_i = (r_{ik} - c_k) n_k + r_{kj} n_j. 
	\]
	Since $r_{ik} - c_k > 0$, we have that $c_i - r_{ki} \geq c_i$ by the minimality of $c_i$, implying that $r_{ki} = 0$. Repeating this argument gives us
	\begin{align*}
		c_i n_i & = r_{ik} n_k,\\
		c_k n_k & = r_{kj} n_j,\\
		c_j n_j & = r_{ji} n_i.
	\end{align*}
	However, the product of all three will yield $c_ic_kc_j = r_{ik}r_{kj}r_{ji}$, but this is impossible as $r_{ik} > c_i$ and similarly $r_{kj} > c_k$ and $r_{ji} > c_j$. Thus, $r_{ij} > 0$. 
\end{proof}

\begin{lemma}\label{Lem:RecoverGeneratorsFromBetti3}
	Let $M = \langle n_1, n_2, n_3 \rangle$ be a numerical monoid of embedding dimension three such that $\abs{\Betti(M)} = 3$. Then
	\begin{align*}
		n_1 & = r_{12}r_{13} + r_{12}r_{23}+r_{13}r_{32}\\
		n_2 & = r_{13}r_{21} + r_{21}r_{23} + r_{23}r_{31}\\
		n_3 & = r_{12}r_{31} + r_{21}r_{32} + r_{31}r_{32}. 
	\end{align*}
\end{lemma}

\begin{proof}
	This follows from the proofs of \cite[Lemmas 1, 2, 3, 4, and 5]{rosales2004numerical}, noting that the assumption that the generators are pairwise coprime is not needed except to obtain the result of \cite[Lemma 1]{rosales2004numerical}, which can be obtained under the assumption that $\abs{\Betti(M)} = 3$ by the previous lemma. 
\end{proof}

This allows us to establish a bijection extending that of Theorem \ref{Thm:CoprimeBijection}. To state the extended bijection, we extend the definition of a 0-matrix found in \cite[p. 492]{rosales2004numerical}.

\begin{definition}
	A \textbf{0-matrix} is a $3\times 3$ matrix of the form
	\[
	A = \begin{pmatrix}
	0 & a_{12} & a_{13} \\
	a_{21} & 0 & a_{23}\\
	a_{31} & a_{32} & 0
	\end{pmatrix},
	\]
	where $a_{12}, a_{13}, a_{21}, a_{23}, a_{31}, a_{32} \in \N \setminus\{0\}$ are such that $\langle A_1, A_2, A_3 \rangle$ is a numerical monoid with embedding dimension three and three Betti elements, with
	\begin{align*}
	A_1 = a_{12}a_{13} + a_{12}a_{23} + a_{13}a_{32},\\
	A_2 = a_{13}a_{21} + a_{21}a_{23} + a_{23}a_{31},\\
	A_3 = a_{12}a_{31} + a_{21}a_{32} + a_{31}a_{32}. 
	\end{align*}
\end{definition}

\begin{theorem}\label{Thm:Bijection}
	Let $X$ be the set consisting of all numerical monoids with embedding dimension three and three Betti elements. Let $Y$ be the set 0-matrices. There exists a bijection between $X$ and $Y$ given by the assignments
	\[
	\langle n_1, n_2, n_3 \rangle \mapsto \begin{pmatrix}
	0 & r_{12} & r_{13} \\
	r_{21} & 0 & r_{23}\\
	r_{31} & r_{32} & 0
	\end{pmatrix} \text{ and } \begin{pmatrix}
	0 & a_{12} & a_{13} \\
	a_{21} & 0 & a_{23}\\
	a_{31} & a_{32} & 0
	\end{pmatrix} \mapsto \langle A_1, A_2, A_3 \rangle . 
	\]

\end{theorem} 
\begin{proof}
	The map from matrices to numerical monoids is injective due to \cite[Theorem 8]{rosales2004numerical}, again noting that the result holds under the weaker condition of three Betti elements. This map surjective due to Lemma \ref{Lem:RecoverGeneratorsFromBetti3}. 
\end{proof}

	Given a numerical monoid $M$ with embedding dimension three and exactly three Betti elements, we can use the bijection in Theorem \ref{Thm:Bijection} to obtain a 0-matrix, which we call the \textbf{0-matrix associated with $M$}.

\section{Longest and Shortest Factorization Lengths}

\indent\indent In this section, we consider numerical monoids $\langle n_1, n_2, n_3 \rangle$ of embedding dimension three with $n_1 < n_2 < n_3$. We first establish a few results to narrow down where exceptions to $L(x + n_1) = L(x) + 1$ or $\ell(x + n_3) = \ell(x) +1$ can occur. We then apply these results separately to numerical monoids of embedding dimension three with three Betti elements and those with at most two Betti elements.

\begin{lemma}\label{Lem:Non1}
	Let $M = \langle n_1, \dots, n_k \rangle$ be a numerical monoid with $n_1 < \cdots < n_k$ of embedding dimension $k$. The followings statements hold. 
	\begin{enumerate}
		\item 	 Suppose for $x \in M$ there exists $(a_1, \dots, a_k) \in \mathsf{Z}(x+n_1)$ with $a_1 > 0$ and $a_1 + \cdots + a_k = L(x+n_1)$. Then $L(x+n_1) = L(x) + 1$.
		\item  Suppose for $x \in M$ there exists $(a_1, \dots, a_k) \in \mathsf{Z}(x+n_k)$ with $a_k > 0$ and $a_1 + \cdots + a_k = \ell(x+n_k)$. Then $\ell(x+n_k) = \ell(x) + 1$.
	\end{enumerate}
\end{lemma}

\begin{proof}
	We always have $L(x) + 1 \leq L(x+n_1)$ since any factorization $(a_1', \dots, a_k')$ of $x$ of the longest length gives us a factorization $(a_1' + 1, a_2' \dots, a_k')$ of $x+n_1$. 
	
	On the other hand, since $a_1 > 0$, the factorization $(a_1, \dots, a_k)$ of $x+n_1$ gives us that $(a_1 - 1, a_2 \dots, a_k)$ is a factorization of $x$. Thus, $L(x+n_1) - 1 \leq L(x)$. 
	
	Combining the two inequalities yields
		\[
	L(x+n_1) = L(x) + 1.
	\]
	
	The proof for the case involving the shortest factorization length is similar.
\end{proof}

Bringing our focus to a numerical monoid $M = \langle n_1, n_2, n_3 \rangle$ of embedding dimension three with $n_1 < n_2 < n_3$, we can find exceptions to $L(x+n_1) = L(x) + 1$ by only considering $x \in M$ such that $x + n_1 \in \langle n_2, n_3 \rangle$. It turns out only one of those elements needs to be tested to see if there exists an $x \in M$ such that $L(x+n_1) \neq L(x) + 1$. Similarly, we look at $x \in M$ such that $x+n_3 \in \langle n_1, n_2 \rangle$ to find exceptions to $\ell(x+n_3) = \ell(x) + 1$. We will first need a definition.

\begin{definition}
	Let $M = \langle n_1, n_2, n_3 \rangle$ with $n_1 < n_2 < n_3$ be a numerical monoid of embedding dimension three. Define
	\[
		\alpha(M) = \min\{b \in \N \mid bn_2 - n_1 \in M \} \text{\quad and \quad} \beta(M) = \min\{b \in \N \mid bn_2 - n_3 \in M \}.
	\]
\end{definition}

This definition allows us to state the following result. 

\begin{theorem}\label{Thm:LongestReduction}
	Let $M = \langle n_1, n_2, n_3 \rangle$ be a numerical monoid of embedding dimension three with $n_1 < n_2 < n_3$. The following statements hold.
	\begin{enumerate}
		\item There exists $x \in M$ with $L(x+n_1) \neq L(x) + 1$ if and only if $L(\alpha(M)n_2) \neq L(\alpha(M)n_2 - n_1) + 1$.
		\item There exists $x \in M$ with $\ell(x+n_3) \neq \ell(x) + 1$ if and only if $\ell(\beta(M)n_2) \neq \ell(\beta(M)n_2-n_3) + 1$.
	\end{enumerate}
\end{theorem}

\begin{proof}
	We only need to prove the forward direction. Suppose that there exists $x \in M$ with $L(x+n_1) \neq L(x) + 1$. Write
	\[
		x+n_1 = an_1 + bn_2 + cn_3,
	\]
	where $a, b, c \in \N$ and $a+b+c = L(x+n_1)$. Then $a = 0$ by Lemma \ref{Lem:Non1}. Since $x \in M$, we also have a factorization
	\[
		x+n_1 = An_1 + Bn_2 + Cn_3,
	\]
	with $A, B, C \in \N$ and $A > 0$. Because $A > 0$, the factorization $(A,B,C)$ of $x+n_1$ cannot be of the longest length. This yields
	\[
		L(x+n_1) = b + c > A + B + C.
	\]
	
	We want to show that $c < C$. Suppose on the contrary that $c \geq C$. Suppose further that $b \geq B$. Then $bn_2 +cn_3 = An_1 + Bn_2 + Cn_3$ becomes
	\[
		(b-B)n_2 + (c-C)n_3 = An_1.
	\]
	Since $n_1 < n_2, n_3$, we have $(b-B) + (c-C) < A$, implying $b+c < A+B+C$, a contradiction. Now consider the case when $c \geq C$ and $b < B$. Then $bn_2 +cn_3 = An_1 + Bn_2 + Cn_3$ becomes
	\[
		(c-C)n_3 = An_1 + (B-b)n_2.
	\]
	This implies $c-C < A + (B-b)$ and therefore $b+c < A+B+C$, a contradiction. 
	
	Therefore, $c < C$, which means $bn_2 = An_1 + Bn_2 + (C-c)n_3$. Since $A > 0$, we have that $bn_2 -n_1 \in M$ and thus $b \geq \alpha(M)$. 
	
	Now write
	\[
		\alpha(M)n_2 - n_1 = a'n_1+b'n_2 +c'n_3
	\]
	with $a', b', c' \in \N$ such that $a'+b'+c' = L(\alpha(M)n_2 - n_1)$. We add $n_1 + (b-\alpha(M))n_2 + cn_3$ to both sides to obtain
	\[
		bn_2 +cn_3 = (a'+1)n_1 +(b'+b-\alpha(M))n_2 + (c' +c)n_3.
	\]
	Because $a'+1 > 0$, we have that $b+c > a' + 1 + b' +b - \alpha(M) + c' +c$, which implies that
	\[
		L(\alpha(M)n_2) \geq \alpha(M) > a' + b' + c' + 1 = L(\alpha(M)n_2 - n_1) + 1.
	\]
	Thus, $L(\alpha(M)n_2) \neq L(\alpha(M)n_2 - n_1) + 1$. The proof for the result about the shortest factorization length is similar. 
\end{proof}

To investigate further, we would like to determine some properties of the invariants $\alpha(M)$ and $\beta(M)$.  

\begin{proposition}\label{Prop:LongestFacts}
	Let $M = \langle n_1, n_2, n_3 \rangle$ be a numerical monoid of embedding dimension three with $n_1 < n_2 < n_3$. Then
	\begin{enumerate}
		\item for all $(a,b,c) \in \mathsf{Z}(\alpha(M)n_2 - n_1)$, we have $b = 0$;
		\item if $L(\alpha(M)n_2)\neq L(\alpha(M)n_2-n_1)+1$, then $L(\alpha(M)n_2) = \alpha(M)$;
		\item $L(\alpha(M)n_2)\neq L(\alpha(M)n_2-n_1)+1$ if and only if for all $(a,b,c) \in \mathsf{Z}(\alpha(M)n_2)$ with $b = 0$, we have $a+c < \alpha(M)$;
		\item if $\alpha(M) = 2$ or $\alpha(M) = n_1$, then $L(\alpha(M)n_2) = L(\alpha(M)n_2-n_1) + 1$; and
		\item $\alpha(M) \leq n_1$. 
	\end{enumerate}

	Analogously, we have
		\begin{enumerate}
		\item for all $(a,b,c) \in \mathsf{Z}(\beta(M)n_2 - n_3)$, we have $b = 0$; 
		\item if $\ell(\beta(M)n_2)\neq \ell(\beta(M)n_2-n_1)+1$, then $\ell(\beta(M)n_2) = \beta(M)$;
		\item $\ell(\beta(M)n_2)\neq \ell(\beta(M)n_2-n_3)+1$ if and only if for all $(a,b,c) \in \mathsf{Z}(\beta(M)n_2)$ with $b = 0$, we have $a+c > \beta(M)$;
		\item if $\beta(M) = n_3$, then $\ell(\beta(M)n_2) = \ell(\beta(M)n_2-n_3) + 1$; and
		\item $\beta(M) \leq n_3$. 
	\end{enumerate}
\end{proposition}

\begin{proof}
	We will prove the statements about $\alpha(M)$. The proofs for the statements about $\beta(M)$ are similar. 
	\begin{enumerate}
\item 	Let $(a,b,c) \in \mathsf{Z}(\alpha(M)n_2 - n_1)$. We have that $\alpha(M)n_2 - n_1 -bn_2 = (\alpha(M)-b)n_2 - n_1 \in M$. Because of the minimality of $\alpha(M)$, we must have $b = 0$. 
	
\item 	Now suppose that $L(\alpha(M)n_2)\neq L(\alpha(M)n_2-n_1)+1$. Let $(A,B,C)$ be a factorization of $\alpha(M)n_2$ such that $A+B+C = L(\alpha(M)n_2)$. Then $A = 0$ by Lemma \ref{Lem:Non1}. We must have $C = 0$ since $C > 0$ will not maximize $B+C$. Thus, $L(\alpha(M)n_2) = \alpha(M)$. 
	
\item 	Still supposing that $L(\alpha(M)n_2)\neq L(\alpha(M)n_2-n_1)+1$, we take $(a,b,c) \in \mathsf{Z}(\alpha(M)n_2)$ with $b = 0$. If $a = 0$, then $cn_3 = \alpha(M)n_2$ implies $a+c = c < \alpha(M)$. If $a > 0$, then $a+c < L(\alpha(M)n_2) = \alpha(M)$ by Lemma \ref{Lem:Non1}. 
	
	On the other hand, suppose $L(\alpha(M)n_2) = L(\alpha(M)n_2-n_1) + 1$. Let $(A,B,C) \in \mathsf{Z}(\alpha(M)n_2-n_1)$ be a factorization such that $A+B+C = L(\alpha(M)n_2-n_1)$. By item 1, $B = 0$. Then $(A+1, 0, C) \in \mathsf{Z}(\alpha(M)n_2)$ and $A+1+B+C = L(\alpha(M)n_2)$. Therefore, $(A+1) + C \geq \alpha(M)$ since $(0,\alpha(M),0) \in \mathsf{Z}(\alpha(M)n_2)$.

\item 	If $\alpha(M) = 2$, then there does not exist $(a, b, c) \in \mathsf{Z}(\alpha(M)n_2)$ with $b =0$ and $a +c < 2$, since this would require $(1,0,0)$ or $(0,0,1) \in \mathsf{Z}(\alpha(M)n_2)$. This implies $L(\alpha(M)n_2) = L(\alpha(M)n_2 - n_1) + 1$ by item 3.
	
	If $\alpha(M) = n_1$, then $(n_2,0,0) \in \mathsf{Z}(\alpha(M)n_2)$ and $n_2 > \alpha(M)$, so $L(\alpha(M)n_2) = L(\alpha(M)n_2 - n_1) + 1$ by item 3.
	
\item $n_1n_2 - n_1 = (n_2-1)n_1 \in M$. 	
\end{enumerate}
\end{proof}

\begin{remark}
	Notice that $\beta(M) = 2$ does not guarantee $\ell(\beta(M)n_2) = \ell(\beta(M)n_2-n_3) + 1$ unlike the analogous statement for $\alpha(M)=2$ in item 4 of Proposition \ref{Prop:LongestFacts}. For example, if $M = \langle 6, 23,34\rangle$, we calculate that $\beta(M) = 2$, but $\ell(2\cdot 23) = 2$ and $\ell(2\cdot 23 - 34) + 1 = 3$.
\end{remark}

Now we focus on numerical monoids with embedding dimension three and three Betti elements. We now make use of the bijection in Theorem \ref{Thm:Bijection}. Using this bijection, we can determine whether a numerical monoid $\langle n_1, n_2, n_3\rangle$ of embedding dimension three with exactly three Betti elements has an exception to $L(x+n_1) = L(x) + 1$ given its associated 0-matrix.

\begin{theorem}\label{Thm:Longest0Matrix}
	Let $M = \langle n_1, n_2, n_3\rangle$ be a numerical monoid of embedding dimension three such that $n_1 < n_2 < n_3$ and $\abs{\Betti(M)} = 3$. Let $A = (r_{ij})$ be the 0-matrix associated with $M$ as in Theorem \ref{Thm:Bijection}. The following statements hold.
	\begin{enumerate}
		\item The identity $L(x+n_1) = L(x)+1$ holds for all $x \in M$ if and only if $r_{12} + r_{32} \leq r_{21} + r_{23}$.
		\item The identity $\ell(x+n_3) = \ell(x) +1$ holds for all $x \in M$ if and only if $r_{12} + r_{32} \geq r_{21} + r_{23}$.
	\end{enumerate}

\end{theorem}

\begin{proof}
	Using Theorem \ref{Thm:LongestReduction} and Proposition \ref{Prop:LongestFacts}(3), there exists $x \in M$ such that $L(x+n_1) \neq L(x) + 1$ if and only if for all $(a,b,c) \in \mathsf{Z}(mn_2)$ with $b = 0$, we have $a+c < \alpha(M)$. Let $c_2$ be the invariant defined in Section \ref{Sec:Betti}. Since $r_{12} > 0$, we have $c_2 = \alpha(M)$. Because $\abs{\Betti(M)} = 3$, Theorem \ref{Thm:Bijection} implies that the only possibility is that $a=r_{21}$ and $c=r_{23}$. We also have $r_{12} + r_{32} = c_2$ by Lemma \ref{Lemma:ColumnAdd}. Thus, $a+c < \alpha(M)$ is equivalent to 
	\[
		r_{12} + r_{32} > r_{21} + r_{23}.
	\]
	
	The proof for the statement about $\ell(x+n_3) = \ell(x) + 1$ is similar.
\end{proof}

This theorem allows us to produce numerical monoids of embedding dimension three with large generators that have exceptions to $L(x+n_1) = L(x) + 1$ or $\ell(x+n_3) = \ell(x) +1$. 

\begin{example}
	We can verify that the 0-matrix
	\[
		\begin{pmatrix}
		0&4321&6\\
		2000&0&89\\
		7001&1234&0
		\end{pmatrix}
	\]
	is associated with the numerical monoid $\langle 417899, 813089, 41358555 \rangle$, 
	which has three Betti elements. Here, $\alpha(M) = 5555$,
	\[
		L(5555\cdot 813089) = 5555, \text{\quad and \quad} L(5555\cdot 813089 - 417899) = 2088,
	\]
	with the latter coming from the factorization $5555\cdot 813089 - 417899 = 1999\cdot 417899 + 89\cdot 41358555$. 

\end{example}

Now assume that $M = \langle n_1, n_2, n_3 \rangle$ is a numerical monoid of embedding dimension three and $\abs{\textup{Betti}(M)} \leq 2$. By Theorem \ref{Thm:2BettiForm}, the monoid $M$ can be written in the form $\langle am_1, am_2, bm_1 + cm_2 \rangle$ for some $a,b,c,m_1,m_2 \in \N$, and whether there is an exception to $L(x+n_1) = L(x) + 1$ or $\ell(x+n_3) = \ell(x) + 1$ can be determined from this form. 

\begin{theorem}\label{Thm:2Betti}
	Let $M = \langle n_1, n_2, n_3 \rangle$ be a numerical monoid of embedding dimension three and $\abs{\textup{Betti}(M)} \leq 2$. Write $M$ in the form
	\[
	M = \langle am_1, am_2, bm_1 + cm_2 \rangle
	\]
	for some $a,b,c, m_1, m_2 \in \N$ such that $ 1 < m_1 < m_2$, $\gcd(m_1, m_2) = 1$, $a\geq 2$, $ b+c \geq 2$, and $\gcd(a, bm_1+cm_2) = 1$. The following statements hold.
	\begin{enumerate}
		\item If $am_1 < am_2 < bm_1 + cm_2$, then $L(x+n_1) = L(x) + 1$ for all $x \in M$. 
		\item If $am_1 < bm_1 + cm_2 < am_2$ and $b \neq 0$, then $L(x+n_1) = L(x) + 1$ for all $x\in M$ if and only if $a \leq b + c + \left\lfloor \frac{c}{m_1} \right\rfloor (m_2 - m_1)$.
		\item If $bm_1 + cm_2 < am_1 < am_2$, then $L(x+n_1) = L(x) + 1$ for all $x\in M$ if and only if $b+ \left\lceil \frac{c}{m_1}\right\rceil m_2 \leq  a + \left\lceil \frac{c}{m_1}\right\rceil m_1 - c$.
		\item If $am_1 < bm_1 + cm_2 < am_2$ and $b = 0$, then $L(x+n_1) = L(x) + 1$ for all $x \in M$ if and only if $\left\lceil \frac{m_1}{c}\right\rceil a \leq \left\lceil \frac{m_1}{c}\right\rceil c + m_2 - m_1$. 
	\end{enumerate}
	For the shortest factorization lengths, the following statements hold.
	\begin{enumerate}
		\item If $am_1 < am_2 < bm_1 + cm_2$, then $\ell(x+n_3) = \ell(x) + 1$ for all $x \in M$ if and only if $c + \left\lceil \frac{b}{m_2}\right\rceil m_1 \geq \left\lceil \frac{b}{m_2}\right\rceil m_2 - b + a$. 
		
		\item If $am_1 < bm_1 + cm_2 < am_2$ and $c \neq 0$, then $\ell(x+n_3) = \ell(x) + 1$ for all $x \in M$ if and only if $a \geq b + c + \left\lfloor\frac{b}{m_2}\right\rfloor (m_1 - m_2)$. 
		
		\item If $bm_1 + cm_2 < am_1 < am_2$, then $\ell(x+n_3) = \ell(x) + 1$ for all $x \in M$.
		
		\item If $am_1 < bm_1 + cm_2 < am_2$ and $c = 0$, then $\ell(x+n_3) = \ell(x) + 1$ for all $x \in M$ if and only if $ \left\lceil \frac{m_2}{b}\right\rceil a \geq \left\lceil \frac{m_2}{b}\right\rceil b - m_2 + m_1$. 
	\end{enumerate}
\end{theorem}

\begin{proof}
	From \cite[p. 582]{garcia2018delta}, since $\abs{\textup{Betti}(M)} \leq 2$, we have that
	\[
		\Betti(M) = \{a(bm_1+cm_2), am_1m_2 \}
	\]
	with
	\begin{align*}
		\mathsf{Z}(a(bm_1 + cm_2)) = &\left\{\left(b-\left\lfloor\frac{b}{m_2}\right\rfloor m_2,c+\left\lfloor\frac{b}{m_2}\right\rfloor m_1, 0 \right), \dots, (b-m_2, c+m_1, 0), (b,c,0), \color{white} \right\}\\ &\color{white} \left\{ \color{black} (b+m_2, c-m_1, 0), \dots, \left(b+\left\lfloor\frac{c}{m_1}\right\rfloor m_2, c- \left\lfloor \frac{c}{m_1} \right\rfloor m_1, 0\right), (0,0,a) \right\}
	\end{align*}
	and
	\[
		\mathsf{Z}(am_1m_2) = \{(m_2, 0,0), (0,m_1,0) \}. 
	\]
	Here, we make no assumptions about the ordering of the numbers $am_1, am_2, bm_1 + cm_2$ yet. A factorization $(a_1, a_2, a_3) \in \mathsf{Z}(x)$ means $a_1 \cdot am_1 + a_2 \cdot am_2 + a_3 \cdot (bm_1 + cm_2) = x$. Once we make assumptions about the ordering of the generators, factorizations will be written so that the corresponding generators are in ascending order. 
	
	Let $c_i$ be the invariant defined in Section \ref{Sec:Betti}. Suppose that $n_i = bm_1 + cm_2$. Then $c_i = a$ by Proposition \ref{Prop:FormOfBettiElts}. Otherwise, $c_i(bm_1 + cm_2) = am_1m_2$ with $c_i < a$, but $a$ divides neither $c_i$ nor $bm_1 + cm_2$.

	Suppose that $n_j = am_1$. Then $c_j = m_2$ by Proposition \ref{Prop:FormOfBettiElts}. If this is not the case, we would have $a(bm_1 + cm_2) = c_j \cdot am_1 < am_1m_2$. Based on the factorizations of $a(bm_1+cm_2)$, this requires $c- \left\lfloor \frac{c}{m_1} \right\rfloor m_1 = 0$, so $m_1 | c$. Thus,
	\[
		a(bm_1+cm_2) = \left(b+\left\lfloor\frac{c}{m_1}\right\rfloor m_2 \right)am_1 > m_2 \cdot am_1,
	\]
	a contradiction. Similarly, if $n_k = am_2$, then $c_k = m_1$. 
	
	Now we analyze each case using Theorem \ref{Thm:LongestReduction} and Proposition \ref{Prop:LongestFacts}(3), comparing $\alpha(M)$ and the length of the longest factorization of $\alpha(M)n_2$ without involving $n_2$. 
	
	\begin{enumerate}
		\item Suppose that $am_1 < am_2 < bm_1 + cm_2$. Then we have $\alpha(M) = m_1$. This is because $c_2 = m_1$ and we have $m_1\cdot am_2 - am_1 = (m_2-1)\cdot am_1 \in M$. Since $\mathsf{Z}(am_1m_2) = \{(m_2, 0,0), (0,m_1,0) \}$, we have that $x \in M$ such that $L(x+n_1) \neq L(x) + 1$ if and only if $m_2 + 0 < m_1$, which never occurs since $m_1 < m_2$. 
		
		\item Now assume that $am_1 < bm_1 + cm_2 < am_2$ and $b \neq 0$. Then $\alpha(M) = a$ due to the facts that $c_2 = a$ and  $\left(b,0,c\right) \in \mathsf{Z}(a(bm_1+cm_2))$ with $b \neq 0$. From the factorizations of $a(bm_1+cm_2)$, the longest factorization of the form $(a_1, 0, a_3)$ is $\left(b+\left\lfloor\frac{c}{m_1}\right\rfloor m_2, 0, c- \left\lfloor \frac{c}{m_1} \right\rfloor m_1\right)$, which has length $b + c + \left\lfloor \frac{c}{m_1} \right\rfloor (m_2 - m_1)$. Therefore, there is an exception to $L(x+n_1) = L(x) + 1$ for all $x \in M$ if and only if $a > b + c + \left\lfloor \frac{c}{m_1} \right\rfloor (m_2 - m_1)$. 
		
		\item In this case, we assume that $bm_1 + cm_2 < am_1 < am_2$. We claim that $\alpha(M) = b+ \left\lceil \frac{c}{m_1}\right\rceil m_2$. We first check that 
		\[
			\left(b+\left\lceil \frac{c}{m_1}\right\rceil m_2 \right)am_1 - a(bm_1+cm_2) = \left(\left\lceil \frac{c}{m_1}\right\rceil m_1 - c\right)am_2 \in M,
		\]
		so there exists $(a_1, a_2, a_3) \in Z\left(\left(b+\left\lceil \frac{c}{m_1}\right\rceil m_2 \right)am_1\right)$ such that $a_1 \neq 0$. On the other hand, there exist $a_1, a_3 \in \N$ with $a_1 > 0$ such that 
		\[\alpha(M)\cdot am_1 = a_1(bm_1+cm_2) + a_3 \cdot am_2.\] 
		Then we have $\alpha(M) m_1 = k(bm_1 + cm_2)  + a_3 m_2 = kbm_1 + (kc+a_3)m_2$, where $k  = \frac{a_1}{a} \in \N \setminus \{0\}$ since $\gcd(a, bm_1 + cm_2) = 1$. Since $\gcd(m_1, m_2) = 1$, we have that \[\mathsf{Z}(\alpha(M)m_1) = \left\{(\alpha(M),0), (\alpha(M)-m_2, m_1), \dots, \left(\alpha(M)-\left\lfloor\frac{\alpha(M)}{m_2}\right\rfloor m_2, \left\lfloor\frac{\alpha(M)}{m_2}\right\rfloor m_1 \right) \right\}\] in $\langle m_1, m_2 \rangle$. Therefore, $\alpha(M) - gm_2 = kb$ and $gm_1 = kc+a_3$ for some $g \in \N$. This implies that $g \geq \frac{c}{m_1}$ so $g \geq \left\lceil \frac{c}{m_1}\right\rceil$. Now we see that $\alpha(M) = kb + gm_2 \geq b + \left\lceil \frac{c}{m_1}\right\rceil m_2$. Thus, $\alpha(M) = b + \left\lceil \frac{c}{m_1}\right\rceil m_2$.

		 We now want to show that $\left(a, 0, \left\lceil \frac{c}{m_1}\right\rceil m_1 -c\right)$ is the longest factorization of $\alpha(M) \cdot am_1$ without involving $am_1$. Suppose again that we have $(a_1,0, a_3) \in \mathsf{Z}(\alpha(M)\cdot am_1)$ with $a_1 > 0$. Then from before, we have $\alpha(M) = kb + gm_2$, where $k  = \frac{a_1}{a} \in \N \setminus \{0\}$ and $g \in \N$ such that $g \geq \left\lceil \frac{c}{m_1}\right\rceil $. Assume for a contradiction that $a_1 > a$. Then $k \geq 2$, so $\alpha(M) > b + \left\lceil \frac{c}{m_1}\right\rceil m_2$, which is impossible. Thus, $\left(a, 0, \left\lceil \frac{c}{m_1}\right\rceil m_1 -c\right)$ is the longest factorization of $\alpha(M) \cdot am_1$ without involving $am_1$, and we can conclude that $L(x+n_1) = L(x) + 1$ for all $x\in M$ if and only if $b+ \left\lceil \frac{c}{m_1}\right\rceil m_2 \leq  a + \left\lceil \frac{c}{m_1}\right\rceil m_1 - c$.

		\item Suppose that $am_1 < bm_1 + cm_2 < am_2$ and $b = 0$. Let $a' = m_2$, $m_1' = c$, $m_2' = a$, $b' = 0$, and $c' = m_1$. We now have $c'm_2' < a'm_1' < a'm_2'$. 
		
		Now we verify that these parameters are valid. We have that $1 < m_1' < m_2'$ since $c m_2 < am_2$. It must be that $\gcd(m_1', m_2') = 1$ because otherwise $\gcd(n_1, n_2, n_3) \neq 1$. Moreover, $a' = m_2 \geq 2$ and $b' + c' = m_1 \geq 2$. Lastly, $\gcd(a', b'm_1' + c'm_2') = \gcd(m_2, am_1) = 1$ since $\gcd(a, cm_2) = 1$ and $\gcd(m_1, m_2)= 1$. 
		
		We are in case 3 with this new form. Therefore, $L(x+n_1) = L(x) + 1$ for all $x \in M$ if and only if $b' +  \left\lceil \frac{c'}{m_1'}\right\rceil m_2' \leq \left\lceil \frac{c'}{m_1'}\right\rceil m_1' + a' - c'$, which happens if and only if
		\[
			\left\lceil \frac{m_1}{c}\right\rceil a \leq \left\lceil \frac{m_1}{c}\right\rceil c + m_2 - m_1. 
		\]
	\end{enumerate}

	For the shortest factorization lengths, we do a similar analysis.
	\begin{enumerate}
		\item Under the assumptions of  $am_1 < am_2 < bm_1 + cm_2$, we calculate $\beta(M) = c + \left\lceil \frac{b}{m_2}\right\rceil m_1$. We know there exist $a_1, a_3 \in \N$ with $a_3 > 0$ such that
		\[
			\beta(M) \cdot am_2 = a_1 \cdot am_1 + a_3  (bm_1 + cm_2). 
		\]
		Due to the fact that $\gcd(a, bm_1+cm_2) = 1$, dividing by $a$ yields
		\[
			\beta(M)m_2 = (a_1+kb) m_1 + kc m_2,
		\]
		where $k = \frac{a_3}{a} \in \N \setminus \{0\}$. Now we have that $\beta(M) - gm_1 = kc$ and $gm_2 = a_1 + kb$ for some $g \in \N$. The latter implies that $g \geq \left\lceil \frac{b}{m_2}\right\rceil$ and thus
		\[
			\beta(M) = kc + gm_1 \geq c + \left\lceil \frac{b}{m_2}\right\rceil m_1.
		\]
		Furthermore, we have
		\[
			\left(c + \left\lceil \frac{b}{m_2}\right\rceil m_1 \right) am_2 - a(bm_1 + cm_2) = \left( \left\lceil \frac{b}{m_2}\right\rceil m_2 - b\right)am_1 \in M, 
		\]
		proving that $\beta(M) = c + \left\lceil \frac{b}{m_2}\right\rceil m_1$. 
		
		Now we claim that $\left(\left\lceil \frac{b}{m_2}\right\rceil m_2 - b, 0, a \right)$ is the shortest factorization of $\beta(M) \cdot am_2$ without involving $am_2$. Suppose we have such a factorization of $\beta(M) \cdot am_2$. Then we have $a_1, a_3 \in \N$ with $a_3 > 0$ such that $\beta(M) \cdot am_2 = a_1 \cdot am_1 + a_3  (bm_1 + cm_2)$. Following the calculations that were previously done, we get that $\beta(M) = kc + gm_1$, where $g \geq \left\lceil \frac{b}{m_2}\right\rceil$ and $k = \frac{a_3}{a}$. If $(a_1, 0, a_3)$ is a strictly shorter factorization than $\left(\left\lceil \frac{b}{m_2}\right\rceil m_2 - b, 0, a \right)$, then $a_3 > a$, so $k \geq 2$. This implies that $\beta(M) = kc + gm_1 > c + \left\lceil \frac{b}{m_2}\right\rceil m_1 = \beta(M)$, a contradiction. Thus, $\ell(x+n_3) = \ell(x) + 1$ for all $x \in M$ if and only if \[c + \left\lceil \frac{b}{m_2}\right\rceil m_1 \geq \left\lceil \frac{b}{m_2}\right\rceil m_2 - b + a.\]
		
		\item Consider the case where $am_1 < bm_1 + cm_2 < am_2$ and $c \neq 0$. Then $c_2 = a$ and $(b,0,c) \in \mathsf{Z}(a(bm_1+cm_2))$ with $c \neq 0$, so $\beta(M) = a$. Now we look at all the factorizations of $a(bm_1+cm_2)$ and see that there is an exception to $\ell(x+n_3) = \ell(x) + 1$ if and only if
		\[
			\beta(M) < b + c + \left\lfloor\frac{b}{m_2}\right\rfloor(m_1 - m_2)
		\]
		since $\left(b-\left\lfloor\frac{b}{m_2}\right\rfloor m_2,0,c+\left\lfloor\frac{b}{m_2}\right\rfloor m_1 \right)$ is the shortest factorization of $a(bm_1+cm_2)$ not involving $bm_1 + cm_2$. 
		
		\item Assuming $bm_1 + cm_2 < am_1 < am_2$, we calculate $\beta(M) = m_2$ since $c_2 = m_2$ and $m_2 \cdot am_1 - am_2 = (m_1-1) am_2 \in M$. Then by the factorizations of $\beta(M) \cdot am_1 = am_1m_2$, we have an exception to $\ell(x+n_3) = \ell(x) + 1$ if and only if $\beta(M) = m_2 < m_1$, which is impossible. 
		
		\item Now suppose that $am_1 < bm_1 + cm_2 < am_2$ and $c = 0$. Then we can check that the parameters $a' =m_1, m_1' = a, m_2' = b, b' = m_2, c' = 0$ are valid parameters that put this numerical monoid in case 1. We then apply the result of case 1. 
	\end{enumerate}
\end{proof}

We now provide some examples of numerical monoids of embedding dimension three with at most two Betti elements that satisfy the inequalities in Theorem \ref{Thm:2Betti} and some numerical monoids that do not satisfy the inequalities, showing that the inequalities in Theorem \ref{Thm:2Betti} are necessary to state. We start with cases for $L(x)$ first. 

\begin{example}
	In this example, we consider the case where $am_1 < bm_1+cm_2 < am_2$ and $b \neq 0$. 
	
	The numerical monoid $\langle 15, 19, 21 \rangle$ has parameters $(a,b,c,m_1,m_2) = (3,1,2,5,7)$. These parameters satisfy $a \leq b + c + \left\lfloor \frac{c}{m_1} \right\rfloor (m_2 - m_1)$, so $L(x+ 15) = L(x) + 1$ for all $x$.
	
	However, $\langle 33,34,55 \rangle$ has parameters $(a,b,c,m_1,m_2) = (11,3,5,3,5)$. These parameters do not satisfy $a \leq b + c + \left\lfloor \frac{c}{m_1} \right\rfloor (m_2 - m_1)$. In fact, $L(11 \cdot 34) = 11$ but $L(11\cdot 34 - 33) + 1 = 10$. 
\end{example}

\begin{example}
	Here, we are looking at $bm_1 + cm_2 < am_1 < am_2$. 
	
	The monoid $\langle 23,33,44 \rangle$ has parameters $(a,b,c,m_1,m_2) = (11,1,5,3,4)$. The parameters satisfy the inequality $b+ \left\lceil \frac{c}{m_1}\right\rceil m_2 \leq  a + \left\lceil \frac{c}{m_1}\right\rceil m_1 - c$, so $L(x + 23) = L(x) + 1$ for all $x$. 
	
	On the other hand, the monoid $\langle 8,9,15 \rangle$ has parameters $(a,b,c,m_1,m_2) = (3,1,1,3,5)$. These parameters do not satisfy $b+ \left\lceil \frac{c}{m_1}\right\rceil m_2 \leq  a + \left\lceil \frac{c}{m_1}\right\rceil m_1 - c$. We have $L(6\cdot 9) = 6$ but $L(6\cdot 9 - 8) + 1 = 5$.
\end{example}

\begin{example}
	Now we consider the case where $am_1 < bm_1 + cm_2 < am_2$ and $b = 0$. 
	
	For the monoid $\langle 15,22,33 \rangle$, we have the parameters $(a,b,c,m_1,m_2) = (3,0,2,5,11)$. This satisfies $\left\lceil \frac{m_1}{c}\right\rceil a \leq \left\lceil \frac{m_1}{c}\right\rceil c + m_2 - m_1$, so $L(x + 15) = L(x) + 1$ for all $x$. 
	
	The monoid $\langle 65,66,143 \rangle$ has parameters $(a,b,c,m_1,m_2) = (13,0,6,5,11)$. This does not satisfy $\left\lceil \frac{m_1}{c}\right\rceil a \leq \left\lceil \frac{m_1}{c}\right\rceil c + m_2 - m_1$. In this monoid, we have $L(13 \cdot 66) = 13$ but $L(13\cdot 66 - 65) + 1 = 12$.
\end{example}

Now we see examples for $\ell(x)$. 

\begin{example}
	We look at $am_1 < am_2 < bm_1 + cm_2$. 
	
	Consider the monoid $\langle 35, 40, 52 \rangle$. This monoid is given by the parameters $(a,b,c,m_1,m_2) = (5,4,3,7,8)$. With these parameters, the inequality $c + \left\lceil \frac{b}{m_2}\right\rceil m_1 \geq \left\lceil \frac{b}{m_2}\right\rceil m_2 - b + a$ holds, so $\ell(x+52) = \ell(x) + 1$ for all $x$. 
	
	Now consider $\langle 35,55,61 \rangle$, which is given by $(a,b,c,m_1,m_2) = (5,4,3,7,11)$. The inequality $c + \left\lceil \frac{b}{m_2}\right\rceil m_1 \geq \left\lceil \frac{b}{m_2}\right\rceil m_2 - b + a$ does not hold. In particular, we have $\ell(10 \cdot 55) = 10$ but $\ell(10 \cdot 55 - 61) + 1 = 12$. 
\end{example}

\begin{example}
	Now consider the case $am_1 < bm_1 + cm_2 < am_2$ and $ c \neq 0$. 
	
	The parameters $(a,b,c,m_1,m_2) = (3,2,1,5,7)$ yield the monoid $\langle 15,17,21 \rangle$. These parameters satisfy $a \geq b + c + \left\lfloor\frac{b}{m_2}\right\rfloor (m_2 - m_1)$, so $\ell(x+21) = \ell(x) + 1$ for all $x$. 
	
	The monoid $\langle 15,23,24 \rangle$ has parameters $(a,b,c,m_1, m_2) = (3,3,1,5,8)$. These parameters do not satisfy $a \geq b + c + \left\lfloor\frac{b}{m_2}\right\rfloor (m_2 - m_1)$. We see that $\ell(3 \cdot 23) = 3$ but $\ell(3 \cdot 23 -24) + 1 = 4$. 
\end{example}

\begin{example}
	Consider the case where $am_1 < bm_1 + cm_2 < am_2$ and $c = 0$. 
	
	We have the monoid $\langle 12,16,21 \rangle$. This monoid is given by the parameters $(a,b,c,m_1, m_2) = (3,4,0,4,7)$. This satisfies $ \left\lceil \frac{m_2}{b}\right\rceil a \geq \left\lceil \frac{m_2}{b}\right\rceil b - m_2 + m_1$. Therefore, $\ell(x+21) = \ell(x) +1$ for all $x$. 
	
	The monoid $\langle 35,75,77\rangle$ is given by the parameters $(a,b,c,m_1,m_2) = (7,15,0,5,11)$. The inequality $ \left\lceil \frac{m_2}{b}\right\rceil a \geq \left\lceil \frac{m_2}{b}\right\rceil b - m_2 + m_1$ does not hold. We have $\ell(7\cdot 75) = 7$ but $\ell(7\cdot 75 - 77) + 1 = 9$. 
\end{example}

Putting the results about numerical monoids of embedding dimension three together, we find that there is only a particular family of numerical monoids $\langle n_1, n_2, n_3 \rangle$ with $n_1 < n_2 < n_3$ of embedding dimension three such that exceptions to both formulas $L(x+n_1) = L(x) + 1$ and $\ell(x+n_3) = \ell(x) + 1$ are found. 

\begin{proposition}
	Let $M = \langle n_1, n_2, n_3 \rangle$ with $n_1< n_2 < n_3$ be a numerical monoid with embedding dimension three. Then $L(x+n_1) = L(x) + 1$ for all $x \in M$ or $\ell(x+n_3) = \ell(x) + 1$ for all $x \in M$.
\end{proposition}

\begin{proof}
	If $M$ has three Betti elements, then an exception to both formulas about the longest and shortest factorization lengths would require $r_{12} + r_{32} > r_{21} + r_{23}$ and $r_{12} + r_{32} < r_{21} + r_{23}$, which is impossible. 
	
	Now we assume that $M$ has at most two Betti elements. Using the notation and results of Theorem \ref{Thm:2Betti}, we see that we must have $am_1 < bm_1+cm_2 < am_2$ in order to have a chance to have $x,y \in M$ such that $L(x+n_1) \neq L(x) + 1$ and $\ell(y+n_3) \neq \ell(y) + 1$. If $b\neq 0$ and $c \neq 0$, then there exist $x,y \in M$ such that $L(x+n_1) \neq L(x) + 1$ and $\ell(y+n_3) \neq \ell(y) + 1$ if and only if 
	\[
		b+c+\left\lfloor\frac{c}{m_1}\right\rfloor(m_2 - m_1) < a < b+c+\left\lfloor\frac{b}{m_2}\right\rfloor(m_1 - m_2) ,
	\]
	but this implies $\left\lfloor\frac{c}{m_1}\right\rfloor(m_2 - m_1) < \left\lfloor\frac{b}{m_2}\right\rfloor(m_1 - m_2)$, which is impossible since the left-hand side is positive and the right-hand side is negative. If $b \neq 0$ and $c = 0$, then there exist $x,y \in M$ such that $L(x+n_1) \neq L(x) + 1$ and $\ell(y+n_3) \neq \ell(y) + 1$ if and only if 
	\[
		a > b+c+\left\lfloor\frac{b}{m_2}\right\rfloor(m_2 - m_1) \text{\quad and \quad}  \left\lceil \frac{m_2}{b}\right\rceil a < \left\lceil \frac{m_2}{b}\right\rceil b - m_2 + m_1.
	\]
	This implies that $\left\lceil \frac{m_2}{b}\right\rceil b - m_2 + m_1 > \left\lceil \frac{m_2}{b}\right\rceil b$, which implies that $m_1 > m_2$, a contradiction. Similarly, assuming $b = 0$, $c\neq 0$, and that there exist $x,y \in M$ such that $L(x+n_1) \neq L(x) + 1$ and $\ell(y+n_3) \neq \ell(y) + 1$ leads to a contradiction. 
\end{proof}

\bibliographystyle{amsplain}
\bibliography{references}

\end{document}